\documentclass{amsart}
\usepackage{graphicx}
\usepackage{amsfonts}
\usepackage{amssymb}
\usepackage{amsmath,amsxtra,amssymb}
\newtheorem{theorem}{Theorem}[section]
\newtheorem{lemma}[theorem]{Lemma}
\newtheorem{proposition}[theorem]{Proposition}
\newtheorem{corollary}[theorem]{Corollary}
\newtheorem{definition}[theorem]{Definition}
\newtheorem{example}[theorem]{Example}

\newtheorem{remark}[theorem]{Remark}

\numberwithin{equation}{section}

\newcommand{\Z}{{\field{Z}}}
\newcommand{\R}{{\field{R}}}
\newcommand{\A}{{\mathcal{A}}}

\newcommand{\field}[1]{\mathbb{#1}}
\newcommand{\C}{{\field{C}}}

\newcommand{\la}{{\langle}}
\newcommand{\ra}{\rightarrow}
\newcommand{\id}{{\iota}}
\newcommand{\ot}{{\otimes}}
\newcommand{\tp}{{\widehat{\otimes}}}
\newcommand{\vtp}{{\overline{\otimes}}}

\newcommand{\mtp}{{\otimes_{min}}}

\newcommand{\N}{{\mathcal{N}}}
\newcommand{\M}{{\mathcal{M}}}
\newcommand{\T}{{\mathcal{T}}}
\newcommand{\B}{{\mathcal{B}}}

\newcommand{\lam}{{\lambda}}
\newcommand{\fee}{{\varphi}}
\newcommand{\si}{{\psi}}
\newcommand{\om}{{\omega}}
\newcommand{\G}{{\field{G}}}
\newcommand{\LL}{{L^\infty(\G)}}
\newcommand{\LO}{{L^1(\G)}}
\newcommand{\LT}{{L^2(\G)}}

\newcommand{\CZ}{{C_0(\G)}}
\newcommand{\CU}{{C_u(\G)}}
\newcommand{\LLL}{{L^\infty(\hat\G)}}

\newcommand{\loneqg}{L^1(\mathbb{G})}
\newcommand{\linfqg}{L^\infty(\mathbb{G})}

\begin{document}

\title{From Quantum Groups to Groups}
\author{Mehrdad Kalantar}

\address{Mehrdad Kalantar \newline School of Mathematics and Statistics,
Carleton University, Ottawa, Ontario, K1S 5B6, Canada}
\email{mkalanta@math.carleton.ca}

\author{Matthias Neufang}

\address{Matthias Neufang \newline School of Mathematics and Statistics,
Carleton University, Ottawa, Ontario, K1S 5B6, Canada}
\email{mneufang@math.carleton.ca}

\vspace{-0.7cm}
\address{${}$ \vspace{-0.4cm}\newline Universit\'{e} Lille 1 -- Sciences et Technologies,
UFR de Math\'{e}matiques, Laboratoire de Math\'{e}matiques Paul Painlev\'{e}  (UMR CNRS 8524),
59655 Villeneuve d'Ascq C\'{e}dex, France}
\email{Matthias.Neufang@math.univ-lille1.fr}
${}$ \\[-5ex]
\address{${}$ \vspace{-0.4cm}\newline The Fields Institute for Research in
Mathematical Sciences, Toronto, Ontario,
Canada M5T 3J1}
\email{mneufang@fields.utoronto.ca}

%\maketitle

\begin{abstract}
In this paper we use the recent developments in the
representation theory of locally compact quantum groups,
to assign, to each locally compact
quantum group $\G$, a locally compact group $\tilde \G$ which
is the quantum version of point-masses, and is an
invariant for the latter. We show that ``quantum point-masses" 
can be identified with several other locally compact groups that can be
naturally assigned to the quantum group $\G$.
 This assignment preserves compactness as well as
discreteness (hence also finiteness), and for large classes of quantum
groups, amenability. We calculate this invariant for some of the most
well-known examples of
non-classical quantum groups.
Also, we show that several structural properties of $\G$ are encoded
by $\tilde \G$: the latter, despite being a simpler object, can carry very
important information about $\G$.
\end{abstract}

\maketitle

\section{Introduction}
\label{intro}
One of Murray and von Neumann's primary motivations to
define and study operator algebras was to study the representation
theory of groups, and in fact, those operator algebras associated
to groups have played a prominent role in the theory of operator algebras ever since.

Also, using different group properties, in constructing and studying various types of
operator algebras, has led to some of the deepest results in the subject.

The aim of this paper is to investigate operator algebras
 associated to locally compact quantum groups $\G$, by studying
 the properties of an assigned locally compact group $\tilde\G$.
 The latter, in the classical case, is the initial group,
 from which those operator algebras are constructed.
 We see that, despite being possibly much smaller object in 
 the non-commutative setting, the group $\tilde\G$ can carry very
important information about the quantum group $\G$.
\par
Locally compact quantum groups, as introduced and 
studied by Kustermans and Vaes in \cite{K-V}, 
provide a category which comprises both classical group algebras 
and group-like objects arising in 
mathematical physics such as Woronowicz's famous quantum group $SU_\mu(2)$.

To find a Pontryagin-type duality theorem which holds for all locally compact groups rather than just abelian ones,
one has to pass to the larger category of locally compact quantum groups.
In order to embed locally compact groups in this larger category,
one has to work with the algebras associated with a group. 
So, instead of working with a locally compact group $G$, we study $L^\infty (G)$, and we consider 
$VN(G)$ as its Pontryagin dual.

Also, there are many other occasions in which one prefers, or even has to pass from a group to
an associated (operator) algebra. But, then there are several equivalent ways to recover
the initial group from these algebras: $G$ is topologically isomorphic, for example, to
\begin{itemize}
\item
the spectrum of $C_0(G)$;
\item
the spectrum of $A(G)$;
\item
the set of all group-like elements in the symmetric quantum group $VN(G)$. %with its canonical co-multiplication;
%\item
%\dots
\end{itemize}

In \cite{Wen} Wendel proved that for a locally compact group $G$, every positive isometric linear 
(left or right) $L^1(G)$-module
map on $L^1(G)$, i.e., every positive isometric (right or left) multiplier, has to be the
convolution by a point-mass. Moreover, the set of point-masses
regarded as maps on $L^1(G)$ with the strong operator (i.e., the
point-norm) topology is homeomorphic to the group $G$. In other words, he showed how
a locally compact group can be recovered from its measure algebra.

In \cite{J-N-R}, Junge, Neufang and Ruan defined an analogue of measure algebra for
locally compact quantum groups, and studied its structure and representation theory; 
in fact,
they investigated the quantum group analogue of the class of completely bounded 
multiplier algebras which play an important role in Fourier analysis over groups, by means of a
representation theorem. The latter result enables the authors
to express quantum group duality precisely in terms of a commutation relation.
When $\G = L^\infty(G)$ for a locally compact group $G$, then the algebra
of completely bounded multipliers $M_{cb} (L^1 (\G))$ defined in \cite{J-N-R} is the measure 
algebra of the group $G$.
So we can regard the
algebra $M_{cb} (L^1 (\G))$ as a quantization of the measure algebra.
This motivated us to look for objects similar to point-masses in the classical case, 
so-to-speak ``quantum point-masses''.

Following this path, in this paper we start with assigning 
to each locally compact quantum group $\G$, a locally compact group $\tilde \G$
 that is an invariant for the latter; we will later prove that 
this assignment preserves compactness as well as discreteness (hence also finiteness), and, for large classes of 
quantum groups, amenability.

We first prove some basic properties of this group before arriving at
one of our main results in section 3, Theorem \ref{322}, establishing identifications between 
several different locally compact
groups which can be assigned to a locally compact quantum group, including the intrinsic group
of the dual quantum group, as well as the spectrum of the universal $C^*$-algebra
and of the $L^1$-algebra of the dual quantum group.

In section 4, we calculate this associated group
for some well-known examples of locally compact quantum groups.
For Woronowicz's class of compact matrix pseudogroups, we always obtain a compact 
Lie group -- which in the case of $SU_\mu(2)$ is precisely the circle group.

In the last section of this paper, we present
various applications of studying this group. 
In particular, we show that for a large class of locally compact quantum groups,
the associated locally compact group cannot be ``small'', and in fact,
the smallness of the latter forces the former to be of a very specific type.
We also see that this group
carries some natural properties inherited from the locally compact quantum group,
which shows that this assignment is natural.

\par
The results in this paper are based on \cite{thesis},
written under the supervision of the second-named author.

\section{Preliminaries}

We recall from  \cite {K-V} and \cite {Vaes1} that a (von
Neumann algebraic) \emph{locally compact quantum group} $\G$ is a
quadruple $(\linfqg, \Gamma, \varphi, \psi)$, where $\linfqg$ is a
von Neumann algebra with a co-associative co-multiplication
$\Gamma: \linfqg\to \linfqg \bar\otimes \linfqg$, $\varphi$
and  $\psi $ are  (normal faithful semi-finite) left and right
Haar weights on $\linfqg$, respectively. 
We write $\M_\fee^+ = \{x\in\LL^+ : \fee(x)<\infty\}$ and
$\N_\fee = \{x\in\LL^+ : \fee(x^*x)<\infty\}$, and we denote
by $\Lambda_\fee$ the inclusion of $\N_\fee$ into the GNS Hilbert
space $L^{2}(\G, \varphi)$ of $\fee$. According to \cite [Proposition 2.11]{K-V},
we can identify $L^{2}(\G, \varphi)$ and $L^{2}(\G, \psi)$, and we simply
use $L^{2}(\G)$ for this Hilbert space in the rest of this paper
 
For each locally compact
quantum group $\G$, there exist a  \emph{left fundamental unitary
operator}  $W$ on $L^{2}(\G)\otimes L^{2}(\G)$
which satisfies  the  \emph{pentagonal relation}
 \begin{equation}
\label {F.pentagonal} W_{12} W_{13}W_{23} = W_{23} W_{12}
\end{equation}
and such that the co-multiplication $\Gamma$ on $\linfqg$ can be expressed as
\begin{equation}
\label {F.com} \Gamma(x) = W^{*}(1\otimes x)W
\quad(x \in \linfqg).
\end{equation}

Let $\loneqg$ be the predual of $\linfqg$. Then the pre-adjoint of
$\Gamma$ induces on $\loneqg$ an associative completely
contractive multiplication
\begin{equation}
\label {F.mul} \star  :  L^{1}(\G)\hat\otimes L^{1}(\G)\ni f_{1} \otimes f_{2}
 \mapsto f_{1} \star f_{2} = (f_{1} \otimes
f_{2})\circ \Gamma \in L^{1}(\G).
\end{equation}

The \emph{left regular representation} $\lambda : L^{1}(\G) \to
\B(L^{2} (\G))$ is defined by
 \[
\lambda : L^{1}(\G)\ni f   \mapsto \lambda(f) = (f\otimes \iota)(W)
\in \B(L^{2}(\G)),
 \]
which is an injective and completely contractive algebra homomorphism
from $L^{1}(\G)$ into $\B(L^{2} (\G))$.
 Then  $L^{\infty} (\hat \G)={\{\lambda(f): f\in \loneqg\}}''$
is the von Neumann algebra associated with the dual quantum group
$\hat \G$ of $\G$.  It follows that $W \in \linfqg \bar \otimes L^{\infty}(\hat \G)$.
We also define the completely contractive injection
\[
\hat\lambda:  L^{1}(\hat \G)\ni\hat f \mapsto \hat\lambda(\hat f) =
(\iota \otimes \hat f)(W)\in \linfqg.
\]
The \emph{reduced quantum group $C^*$-algebra}  
 $C_{0}(\G) = \overline{\hat\lambda(L_{1}(\hat \G) )}^{\|\cdot\|}$
  is a weak$^*$ dense $C^*$-subalgebra of $\LL$ with the 
 co-multiplication
\[
\Gamma : C_{0}(\G)\to M(C_{0}(\G)\otimes
C_{0}(\G))
\]
given by the  restriction of   the co-multiplication on $\LL$  to $C_{0}(\G)$,
where we denote by $M(C_{0}(\G)\otimes C_{0}(\G))$  the multiplier $C^*$-algebra of
the minimal $C^*$-algebra tensor product $C_{0}(\G)\otimes C_{0}(\G)$. 

Let $M(\G)$ denote the operator dual $C_{0}(\G)^{*}$.
There exists  a completely contractive multiplication on $M(\G)$ given by
the convolution
 \[
 \star : M(\G)\tp M(\G)\ni \mu\ot \nu \mapsto \mu \star \nu 
 = \mu (\id\otimes \nu)\Gamma = \nu (\mu \otimes \id)\Gamma
 \in M(\G)
 \]
such that  $M(\G)$ contains $\loneqg$ as a norm closed two-sided ideal.
 If $G$ is a locally compact group, then $C_{0}(\G_{a})$ is the
$C^*$-algebra $C_{0}(G)$ of continuous functions on $G$ vanishing at
infinity, and $M(\G_{a})$ is the measure algebra $M(G)$ of $G$.
Correspondingly, $C_{0}(\hat \G_{a})$ is the left reduced group $C^*$-algebra
$C^{*}_{\lambda}(G)$ of $G$. Hence, we have $M(\hat \G_{a})= B_{\lambda}(G)$.

For a locally compact quantum group $\G$, we
 denote by $S$ its {\it antipode}, which is the unique $\sigma$-strong$^*$
 closed linear map on $\LL$ 
satisfying $(\om \ot\id)(W) \in \mathcal{D}(S)$ for all 
$\om\in\B(\LT)_*$, and $S(\om \ot\id)(W) = (\om \ot\id)(W^*)$,
and such that the elements $(\om \ot\id)(W)$ form a $\sigma$-strong$^*$
 core for $S$. The antipode $S$
has a polar decomposition $S = R\tau_{-\frac i2}$,
 where $R$ is an anti-automorphism of $\LL$ and $(\tau_t)$
is a strongly continuous one-parameter group of automorphisms of $\LL$.
 We call $R$ the {\it unitary antipode} and $(\tau_t)$ the {\it scaling group} of $\G$.

There exist a strictly positive operator $\delta$,
affiliated with $M$, called \emph{modular element}, such that
$$\psi(x) = \varphi(\delta^{1/2} x \delta^{1/2})$$
for all $x\in \M_{\psi}$.
Also, we have $\Gamma(\delta^{it}) = \delta^{it} \otimes \delta^{it}$ for all $t\in \field{R}$.
We say that a locally compact quantum group
$\G$ is {\it unimodular} if $\delta = 1$.

If we define a strictly positive operator $P$ on $\LT$ such that
$P^{it}\Lambda_\fee(x) = \Lambda_\fee(\tau_t(x))$ for all
$t\in\mathbb{R}$ and $x\in\N_\fee$, then we have $\tau_t(x) = P^{it} x P^{-it}$, and
%One has the following equation linking the above-mentioned objects.
\begin{equation}\label{126}
P = \hat P  \  \ \ \ \text{and} \ \ \ \ 
\Delta_{\hat\fee}^{it} = P^{it} J_\fee\delta^{it} J_\fee,
\end{equation}
where $\Delta_{\hat\fee}$ is the modular operator
associated to $\hat\fee$ and $J_\fee$ is the modular conjugate associated to $\fee$.

Moreover, the following hold for all $t\in\field{R}$.
\begin{eqnarray}
\Gamma\circ{\tau_{t}} &=& (\tau_{t} \otimes \tau_{t})\circ\Gamma 
\ \ \ \ \ \ \ \ \ \ \, \ \ \ \ \ \ \ 
\Gamma\circ{\sigma^\fee_{t}} = (\tau_{t} \otimes \sigma^\fee_{t})\circ\Gamma\ ; \label{1.5}\\
\Gamma\circ{\tau_{t}} &=& (\sigma^\fee_{t} \otimes \sigma^\psi_{-t})\circ\Gamma 
\ \ \ \ \ \ \  \ \ \ \ \ \ \
\Gamma\circ{\sigma^\psi_{t}} = (\sigma^\psi_{t} \otimes \tau_{-t})\circ\Gamma \ .\label{1.6}
\end{eqnarray}
For more details,  we refer the reader to \cite{K-V}.

It is known that a locally compact quantum group $\G = (M,\Gamma,\fee,\psi)$ 
is a {\it Kac algebra} if and only if the antipode
$S$ is bounded and the modular element $\delta$ is affiliated with the center of $M$.

Let $L_{*}^1(\hat\G) = \{\hat\om\in L^{1}(\hat \G) : \exists \hat f\in L^{1}(\hat \G) 
\ \text{such that} \ \hat\lambda(\hat\om)^* = \hat\lambda(\hat f)\}$.
Then $L_{*}^1(\hat\G)\subseteq L^{1}(\hat \G)$ is norm dense, and with the involution
$\hat\om^* = \hat f$, and the norm 
$\|\hat\om\|_u = max\{\|\hat\om\| , \|\hat\om^*\|\}$, the space $L_{*}^1(\hat\G)$
becomes a Banach $^*$-algebra (for details see \cite {kus}).
We obtain the   \emph{universal quantum group $C^*$-algebra} $\CU$
as the universal enveloping $C^*$-algebra of the Banach algebra $L_{*}^1(\hat\G)$.
There is a universal $^*$-representation
\[
\hat \lambda_{u} : L_{*}^1(\hat \G)\to \B(H_{u})
\]
such that $\CU = \overline {\hat \lambda_{u}(L^{1}(\hat \G))}^{\|\cdot\|}$.
 There is  a universal co-multiplication
\[
\Gamma_u :\CU\ra M(\CU\otimes\CU), 
\]
and the operator dual $M_u(\G):= C_{u}(\G)^{*}$, which can be regarded
as the space of all {\it quantum measures} on $\G$, is a unital completely contractive Banach 
algebra with multiplication given by
\[
\om \star_{u} \mu = \om (\id\otimes \mu)\Gamma_{u} = \mu(\om\otimes \id)\Gamma_{u}.
\]
By universal property of $\CU$, there is a  unique surjective $^*$-homomorphism 
$\pi:\CU\ra\CZ$  such that $\pi(\hat\lambda_u(\hat\om)) = \hat\lambda(\hat\om)$
for all $\hat\om\in L_{*}^1(\hat \G)$ 
(see \cite {BS}, \cite {B-T} and \cite {kus}).

%%%
%%%%

A linear map $m$ on $\LO$ is called a \emph{left centralizer} 
of $\LO$ if it satisfies
\[
 m(f\star g) = m(f)\star g
\]
for all $f, g \in \LO$.
We denote by $C_{cb}^{l}(L_{1}(\G))$ the space of all
completely bounded left   centralizers of  $\loneqg.$
We have the natural inclusions
\begin{equation}
\label{F.inclusion}
\LO \hookrightarrow M(\G) \hookrightarrow M_u(\G)\to C^{l}_{cb}(\LO).
\end{equation}
These algebras are typically not equal.
We have 
\begin{equation}
\label {F.equi}
M(\G) =  M_u(\G) = C^{l}_{cb}(L_{1}(\G))
\end{equation}
if and only if $\G$ is co-amenable, i.e., $\loneqg$  has a contractive  (or bounded)
approximate identity (cf. \cite {BS}, \cite {B-T}, and \cite {HNR}).
If $\G_{a}$ is the  commutative quantum group associated with a locally compact group $G$, then 
$\G_{a}$ is always co-amenable since 
$L_{1}(\G_{a})= L_{1}(G)$ has a contractive  approximate identity.
On the other hand, if $\G_{s}=\hat \G_{a}$ is the co-commutative 
dual quantum group of $\G_{a}$,  it is co-amenable, i.e., the Fourier algebra $A(G)$ 
has a contractive approximate identity, 
 if and only if the group $G$ is amenable.

The left regular representation $\lambda$ can be extended to $C^{l}_{cb}(\LO)$ such that 
$\langle \hat \om , \lambda(\om) \rangle = \langle \hat\lambda_u(\hat\om) , \om\rangle$
for all $\om\in M_u(\G)$ and $\hat\om \in L_*^1(\hat\G)$.

A normal completely bounded map $\Phi$ on $\LL$ is called \emph{covariant}
if it satisfies
\[
(\id\ot\Phi) \circ \Gamma = \Gamma \circ \Phi.
\]
We denote by ${\mathcal CB}^\sigma_{cov}(\LL)$ the algebra of all normal completely bounded
covariant maps on $\LL$.
It is easy to see that a normal completely bounded map $\Phi$ on $\LL$
is in ${\mathcal{CB}^\sigma_{cov}}(\LL)$ if and only if it is a 
left $L^1(\G)$-module map on $\LL$.
Therefore, a map $T$ is in $C_{cb}^l(L^1(\G))$ if and only if 
$T^*$ is in $\mathcal{CB}^\sigma_{cov}(\LL)$.

Let $X , Y\subseteq \B(H)$. We denote by $\mathcal {CB}^{\sigma, X}_Y(\B(H))$
the algebra of all
normal completely bounded $Y$-bimodule maps $\Phi$ on
$\B(H)$ that leave $X$ invariant.
It was proved in \cite{J-N-R} that for a locally compact quantum group $\G$,
every left centralizer $T\in C_{cb}^l(L^1(\G))$ has a unique
extension to a map $\Phi_T\in \mathcal {CB}^{\sigma, \LL}_\LLL(\B(\LT))$.
The next result, which can be seen as a quantum version of
 Wendel's theorem (\cite[Theorem 3]{Wen}),
is the starting point for our work. In fact, this
theorem will suggest how to define the objects that should be called
quantum point-masses.

In the sequel we shall denote by $Ad(u)$ the map $x\mapsto uxu^*$, for a unitary operator $u$.
\begin{theorem}\label{313}\cite[Theorem 4.7]{J-N-R}
Let $\G$ be a locally compact quantum group, and let $T$ be a complete
contraction in $C_{cb}^l(L^1(\G))$. Then the following are equivalent:
\begin{enumerate}
\item
$T$ is a completely isometric linear isomorphism on $L^1(\G)$;
\item
$T$ has a completely contractive inverse in $C_{cb}^l(L^1(\G))$;
\item
$\Phi_T$ is a $^*$-automorphism in $\mathcal {CB}^{\sigma, \LL}_\LLL(\B(\LT))$;
\item
there exist a unitary operator $\hat u \in \LL$ and a complex number $\lambda\in\field T$
such that $\Phi_T(x) = \lambda Ad(\hat u)(x)$.
If, in addition, $T$ is completely positive, then so is $T^{-1}$. In this case, we have
$\Phi_T = Ad(\hat u)$.
\end{enumerate}
\end{theorem}

\section{Assigning a Group to a Quantum Group}
In view of Theorem \ref{313} and \cite[Theorem 3]{Wen}, we define the following assignment.
\begin{definition}
Let $\G$ be a locally compact quantum group. 
Define $\tilde\G$ to be the set of all completely positive maps
$m\in C_{cb}^l(L^1(\G))$ which satisfy one of the equivalent conditions of
Theorem \ref{313}. We endow $\tilde\G$ with the strong operator topology.
\end{definition}
\begin{example}\label{124365}
 As a consequence of \cite[Theorem 3]{Wen}, we see that
if $\G = L^\infty(G)$ for a locally compact group
$G$, then $\tilde\G$ is topologically isomorphic to $G$.
\end{example}
\begin{example} \cite[Theorem 2]{Ren} If $\G = VN(G)$ for an amenable 
locally compact group $G$, then
$\tilde\G$ is topologically isomorphic to $\hat G$, the set of all continuous characters
on $G$ with the compact-open topology.

As we shall see later, amenability is not necessary.
\end{example}

We thus obtain an assignment $\G \longrightarrow \tilde\G$, from the
category of locally compact quantum groups to the category of groups,
which is inverse to the usual embedding of the latter category into the former.
The main purpose of this paper is to investigate how much information about $\G$
one can get from studying $\tilde\G$, and also study the preservation of several
natural properties under this assignment

Note that in the classical case, for $m \in \tilde\G$, the adjoint map
$m^*:L^\infty(G)\rightarrow L^\infty(G)$ is just the left translation.
The next proposition shows that for a locally compact quantum group $\G$
and $m \in \tilde\G$, the adjoint map 
$m^*$ can be regarded as quantum left translation.
\begin{proposition}\label{315}  Let $m \in \tilde\G$ and $\fee$ be the left
Haar weight on $\G$. Then we have $\fee \circ m^* = \fee$.
\end{proposition}
\begin{proof}  For $x\in \mathcal{M}_\fee^+$ we have 
$m^*(x)\in \mathcal{M}_\fee^+$ as well, and
\begin{eqnarray*}
 \fee \circ m^*(x)1 &=&  (\id\ot\fee)\Gamma(m^*(x)) = (\id\ot\fee)(m^*\ot\id)\Gamma(x)\\
&=& m^*(\id\ot\fee)\Gamma(x) =  m^*(\fee(x)1) = \fee(x)1.
\end{eqnarray*}
\end{proof}
\begin{proposition} \label{316} 
$\tilde\G$ is a topological group.
\end{proposition}
\begin{proof}
Let $m_\alpha \rightarrow \iota$ and $n_\alpha\rightarrow \iota$,
where $(m_\alpha)$ and $(n_\alpha)$ are nets in $\tilde\G$. Then, for all $f\in\LO$, we have
\begin{eqnarray*}
\|(m_\alpha n_\alpha - \iota)(f)\| &\leq& \|(m_\alpha n_\alpha - m_\alpha)(f)\| + 
\|(m_\alpha - \iota)(f)\|\\
&\leq& \|m_\alpha\| \|(n_\alpha - \iota)(f)\| + \|(m_\alpha n_\alpha - \iota)(f)\| \ra 0.
\end{eqnarray*}
We also have
\begin{eqnarray*}
 \|m_\alpha^{-1} (f) - f\| &=& \|m_\alpha^{-1} (f) - m_\alpha^{-1}(m_\alpha (f))\| \\ &=& 
\|m_\alpha^{-1} (f - m_\alpha (f))\| \leq \|f - m_\alpha(f)\| \ra 0.
\end{eqnarray*}
\end{proof}

%\begin{defi}
 Let $(M,\Gamma)$ be a Hopf-von Neumann algebra. 
 The intrinsic group of $(M,\Gamma)$ is defined as
$$Gr(M,\Gamma) := \{x\in M : \Gamma(x) = x\ot x \ \text{and $x$ is invertible} \}.$$
We endow this group with the induced weak$^*$ topology.
It can be easily seen (cf. \cite[Proposition 1.2.3]{E-S}) that
each $x\in Gr(M,\Gamma)$ is in fact a unitary.
For a locally compact quantum group $\G = (M,\Gamma,\fee,\si)$ we denote $Gr(M,\Gamma)$
simply by $Gr(\G)$. For $x\in Gr(\G)$, by \cite[Proposition 5.33]{K-V}, we have
$x\in {\mathcal D}(S)$, and $S(x) = x^*$.
This implies that $\tau_{-i}(x) = S^2(x) = x$,
and it follows that $\tau_t(x) = x$ for all $x\in Gr(\G)$ and $t\in\R$.
Hence, we obtain $R(x) = x^*$ for all $x \in Gr(\G)$.

Using the left fundamental unitary $W$, we can
define the map 
$$\tilde\Gamma: \B(\LT)\ni x\mapsto W^*(1\ot x)W\in\B(\LT)\vtp\B(\LT),$$
which extends the co-multiplication $\Gamma$ to $\B(\LT)$.
%We shall denote this extension also by $\Gamma$.

\begin{proposition}\label{319} We have
$Gr (\G) = Gr (\B(\LT),\tilde\Gamma)$.
\end{proposition}
\begin{proof} Obviously $Gr (\G) \subseteq Gr (\B(\LT),\tilde\Gamma)$. 
If ${x} \in Gr(\B(\LT),\tilde\Gamma)$,
then ${\tilde\Gamma}({x}) = {x} \otimes {x}$; but since
${\tilde\Gamma}(\B(\LT)) \subseteq \LL\vtp\B(\LT)$, we have ${x} \in \LL$.
\end{proof}

The following theorem was proved by de Canni\`{e}re \cite[Definition 2.5]{De-Cann1}
for the case of Kac algebras.
\begin{theorem}\label{317}
Let $\G$ be a locally compact quantum group.
Then we have a group homeomorphism
$$\tilde\G \cong Gr(\hat\G)\,.$$
\end{theorem}
\begin{proof} Let $m\in\tilde\G$. Since $\fee \circ m^* = \fee$ (by Proposition \ref{315}),
we can extend the map
$$\Lambda_{\fee}(x) \mapsto \Lambda_{\fee}(m^*(x)) \ \ (x\in \N_{\fee})$$
to a unitary $\hat u_m$ on $L^2(\G)$, such that $m^* = Ad(\hat u_m)$,
where $Ad(\hat u_m)(x) = u_m x u_m^*$, $x\in\B(\LT)$.
We show that $\hat u_m \in Gr(\hat\G)$. For all $x,y\in \N_{\fee}$ we have
\begin{eqnarray*}
&&(\hat u_m\ot 1) W^* (\Lambda_{\fee}(x)\otimes\Lambda_{\fee}(y)) = 
(\hat u_m\ot 1)\Lambda_{\fee\otimes\fee}(\Gamma(y)(x \otimes 1))\\
&=& \Lambda_{\fee\otimes\fee}((m^*\ot\id)(\Gamma(y)(x\otimes 1))) = 
 \Lambda_{\fee\otimes\fee} (\Gamma(m^*(y)) (m^*(x)\ot 1))\\
&=& W^* (\Lambda_{\fee}(m^*(x)) \otimes \Lambda_{\fee} (m^*(y))) 
= W^* (\hat u_m\otimes \hat u_m) (\Lambda_{\fee}(x) \otimes \Lambda_{\fee}(y)). 
\end{eqnarray*}
Hence, we obtain $W (\hat u_m\ot 1) W^* = \hat u_m\otimes \hat u_m$ which implies
\begin{eqnarray*}
\hat W^* (1\ot\hat u_m) \hat W &=& \chi (W)  (1\ot\hat u_m) \chi (W^*)
= \chi (W (\hat u_m\ot 1) W^*) \\
&=& \chi(\hat u_m\otimes \hat u_m) 
= \hat u_m\otimes \hat u_m.
\end{eqnarray*}
Thus $\hat u_m\in Gr (\B(\LT),\hat\Gamma)$, and so $\hat u_m\in Gr(\hat\G)$ 
by Proposition \ref{319}.

Now define $\Psi: \tilde\G\ra Gr(\hat\G)$, $\Psi(m) = \hat u_m$.
It is easily seen that $\Psi$ is a well-defined group homomorphism.
If $\Psi(m) = 1$, we have $m^* = Ad 1 = \iota$, which implies $m=\iota$.
 Hence $\Psi$ is injective.

To see that $\Psi$ is also surjective, let $\hat u\in Gr(\hat\G)$. 
Then the above calculations
show that for all $\omega\in \B(\LT)_*$, we have
\begin{eqnarray*}
 Ad(\hat u)((\id\ot\omega)W^*) &=& (\id\ot\omega)((\hat u\ot 1)W^*(\hat u^*\ot 1))
= (\id\ot\omega)(W^*(\hat u\ot\hat u)(\hat u^*\ot 1))\\
&=& (\id\ot\omega)(W^*(1\ot\hat u))
= (\id\ot(\hat u\cdot\om))W^*.
\end{eqnarray*}
Since $\{(\id\ot\omega)W^* : \omega\in \B(\LT)_*\}$ 
is weak$^*$ dense in $\LL$, we see that
$$Ad(\hat u)(\LL)\subseteq\LL,$$
whence $Ad(\hat u)\in \mathcal{CB}^{\sigma,\LL}_{\LLL'}(\B(\LT))$.
Hence, by Theorem \ref{313}, there exists
$m\in\tilde\G$ such that $m^* = Ad(\hat u)$, which implies that $\Psi(m) = \hat u$.

We now show that $\Psi$ is a homeomorphism with respect to the corresponding topologies.
For all $x\in\N_\fee$, $y\in\LL$ and $m\in\tilde\G$, and $m^* = Ad(\hat u_m)$, we have
\begin{eqnarray*}
\langle m(\omega_{\Lambda_\fee(x)}), y\rangle &=& \la \omega_{\Lambda_\fee(x)}, m^*(y)\rangle
\langle \omega_{\Lambda_\fee(x)}, \hat u_m y\hat u^*_m\rangle =
\langle \hat u_m y \hat u^*_m\Lambda_\fee(x),\Lambda_\fee(x)\rangle\\ &=&
\langle y \hat u^*_m\Lambda_\fee(x),\hat u^*_m\Lambda_\fee(x)\rangle =
\langle \omega_{\hat u_m^*\Lambda_\fee(x)} , y\rangle,
\end{eqnarray*}
which implies $m(\omega_{\Lambda_\fee(x)}) = \omega_{\hat u^*_m\Lambda_\fee(x)}$.
Now, let $(m_\alpha)$ be a net in $\G$ such that $m_\alpha\ra1$, 
with $m_\alpha^* = Ad(\hat u_\alpha)$.
Then the density of $\Lambda_\fee(\N_\fee)$ in $\LT$ yields
\begin{eqnarray*}
m_\alpha\ra 1 &\Leftrightarrow& \|m_\alpha(\omega_{\Lambda_\fee(x)}) -
 \omega_{\Lambda_\fee(x)}\|\ra 0\ \  \forall x\in\N_\fee \\
&\Leftrightarrow& \|\omega_{\hat u_\alpha^*\Lambda_\fee(x)} - 
\omega_{\Lambda_\fee(x)}\|\ra 0\ \ \forall x\in\N_\fee\\
&\Leftrightarrow& \|\hat u_\alpha^*\Lambda_\fee(x) - 
\Lambda_\fee(x)\|\ra 0\ \ \forall x\in\N_\fee\\
&\Leftrightarrow& \hat u_\alpha^*\stackrel{sot}\longrightarrow 1 
\Leftrightarrow \hat u_\alpha\stackrel{sot}\longrightarrow 1
\Leftrightarrow \hat u_\alpha\stackrel{w^*}\longrightarrow 1.
\end{eqnarray*}
\end{proof}

Our next result is a generalization of the Heisenberg commutation relation,
which has been known to hold for Kac algebras \cite[Corollary 4.6.6]{E-S}.
\begin{theorem}\label{317.5}
Let $\G$ be a locally compact quantum group.
 Let $u\in Gr(\G)$ and $\hat u\in Gr(\hat\G)$. Then there exists
$\lambda\in\field T$ such that
\[
 u\hat u = \lambda\hat u u.
\]
\end{theorem}
\begin{proof}
We obtain from Theorems \ref{313}, \ref{317} and \cite[Theorem 5.1 and Corollary 5.3]{J-N-R}
that $Ad(u)$ and $Ad(\hat u)$ commute as operators on $\B(\LT)$. Therefore we have
\[
 Ad(u)Ad(\hat u) = Ad(\hat u)Ad(u)\Rightarrow Ad(u\hat uu^*\hat u^*) = \iota
\Rightarrow u\hat u u^* \hat u^*\in\C1,
\]
which yields the conclusion.
\end{proof}

Next theorem is as well a generalization of a 
result known in the Kac algebra case \cite[Theorem 3.6.10]{E-S}.
The latter proof uses boundedness of the antipode (which does not hold in the
case of general locally compact quantum groups) in an essential way.
In \cite[Theorem 3.2.11]{thesis} we proved this result for
the general case of locally compact quantum groups.
Here we present a new proof which is also shorter.

For a Banach algebra $\A$ we denote by $sp(\A)$ its spectrum, i.e,
the set of all non-zero bounded multiplicative linear functionals on $\A$.
\begin{theorem}\label{318}
Let $\G$ be a locally compact quantum group. Then we have 
$$Gr (\G) = sp(L^1({\G})).$$
\end{theorem}
\begin{proof} Let $x\in Gr(\G)$. Then ${\Gamma}({x}) = {x} \otimes {x}$ and $x\neq 0$,
and for $\omega,\omega'\in\LO$ we have
\[
 \la \omega\star\omega', x\rangle = \la \omega\ot\omega', \Gamma(x)\rangle = 
\la \omega\ot\omega', x\ot x\rangle = \la \omega, x\rangle \la \omega', x\rangle,
\]
which implies $x\in sp(L^1({\G}))$. Hence, $Gr(\G) \subseteq sp(L^1({\G}))$.

To show the inverse inclusion, first let $x\in sp(\LO)\cap\LL^+$. 
Then $x^{is}\in Gr(\G)$ for all $s\in\R$,
and so $\tau_t(x^{is})= x^{is}$ for all $s,t\in\R$. Therefore, by the equations (\ref{1.6}),
there exists $c\in\C$ such that 
$$\sigma_t^\fee(x)^{is} = \sigma_t^\fee(x^{is}) = c^sx^{is} = (c^{-i}x)^{is}$$ 
for all $s,t\in\R$,
which implies $\sigma_t^\fee(x) = c^{-i}x$ for all $t\in\R$. 
Since $x\geq 0$ and $\sigma_t^\fee$ in a
$^*$-automorphism, we have $\sigma_t^\fee(x) = x$ for all $t\in\R$.
This yields, by \cite[Theorem 2.6]{tak2}, that $x$ is a multiplier of $\M_\fee$, and
we have $\fee(ax) = \fee(xa)$ for all $a\in\M_\fee$. 

Since $\fee$ is n.s.f.,
there exists $a\in\M_\fee$ such that $\fee(ax) = 1$, and we then have
\[
 1 = \fee(ax)1 = (\id\ot\fee)\Gamma(ax) = (\id\ot\fee)(\Gamma(a)(x\ot x)) = 
 \big((\id\ot\fee)(\Gamma(a)(1\ot x))\big)x.
\]
Hence, $x$ has a left inverse. 
Similarly we can show that $x$ has also a right inverse, and therefore
$x$ is invertible.

Now, let $x\in sp(\LO)$. Then by the above we can conclude that both $xx^*$ and $x^*x$ are
invertible, and hence, $x$ is invertible.
\end{proof}
If $\tilde\Gamma: \B(\LT)\ra\B(\LT)\ot\B(\LT)$ is the extension of $\Gamma$ via
the fundamental unitary $W$, then $\tilde\Gamma_*$ defines a product on $\B(\LT)_*$, which
turns the latter to a completely contractive Banach algebra. We denote this
Banach algebra by $\T_\star({\G})$.
\begin{corollary}\label{320}
We have
 $Gr(\B(\LT),{\tilde\Gamma}) =  sp(\T_\star({\G}))$.
\end{corollary}
\begin{proof}
The inclusion $Gr(\B(\LT),{\Gamma})\subseteq sp(\T_\star({\G}))$ is obvious. 
To see the converse, let $x\in sp(\T_\star(\G))$, i.e., $\Gamma(x) = x\ot x$ and $x\neq 0$.
Then we have $x\in\LL$, and so, by Theorem \ref{318}, $x\in Gr(\G)$, which is
equal to $Gr(\B(\LT),\Gamma)$ by Proposition \ref{319}.
\end{proof}
\begin{theorem}\label{321}
Let $\G$ be a locally compact quantum group. Then we have a group homeomorphism
\[
 sp(L^1(\hat\G)) \cong sp(\CU).
\]
\end{theorem}
\begin{proof} 
Let $\phi\in sp(\CU)$. Then for all $\hat\om_1, \hat\om_2\in L_*^1(\hat\G)$ we have
\begin{eqnarray*}
 \la \hat\om_1 \star \hat\om_2 , \lam(\phi)\rangle &=&  
 \la {\hat\lam}_u(\hat\om_1 \star \hat\om_2) , \phi\rangle
= \la {\hat\lam}_u(\hat\om_1) {\hat\lam}_u(\hat\om_2) , \phi\rangle
\\ &=& \la {\hat\lam}_u(\hat\om_1), \phi\rangle \la {\hat\lam}_u(\hat\om_2) , \phi\rangle
= \la \hat\om_1, \lam(\phi)\rangle \la \hat\om_2 , \lam(\phi)\rangle.
\end{eqnarray*}
Since $L_*^1(\hat\G)$ is norm dense in $L^1(\hat\G)$, and
$\langle L^1(\hat\G)\star L^1(\hat\G)\rangle$ is norm dense in $L^1(\hat\G)$,
we have that $\langle L_*^1(\hat\G)\star L_*^1(\hat\G)\rangle$
 is norm dense in $L^1(\hat\G)$.
Hence, $\lam(\phi)\in sp(L^1(\hat\G))$.

Now, let $\hat x \in sp(L^1(\hat\G))$. Then $x\in Gr(\hat\G)$, by
Theorem \ref{318}, and so we have $\hat S(\hat x) = {\hat x}^*$.
Therefore we get $\langle\hat\om^*,\hat x\rangle = 
\overline{\langle\hat\om,\hat S(\hat x)^*\rangle} = 
\overline{\langle\hat\om,\hat x\rangle}$ for all $\hat\om\in L_*^1(\hat\G)$.
Hence, the map $\langle \cdot , \hat x\rangle : L_*^1(\hat\G)\ra\C$
is a non-zero $^*$-homomorphism,
and so by the universality of $\CU$, we obtain a $^*$-homomorphism
$\theta_{\hat x} : \CU\ra\C$ such that 
$\langle {\hat\lam}_u(\hat\om) , \theta_{\hat x}\rangle = \langle\hat\om,\hat x\rangle$
for all $\hat\om\in L_*^1(\hat\G)$.

We show that the induced maps $sp(\CU)\ni\phi\mapsto\lam(\phi)\in sp(L^1(\hat\G))$ and
 $sp(L^1(\hat\G))\ni\hat x\mapsto\theta_{\hat x}\in sp(\CU)$ 
are inverses
to each other. Let $\hat x\in sp(L^1(\hat\G))$, then we have
$$\langle \hat\om , \lam(\theta_{\hat x})\rangle = 
\langle {\hat\lam}_u(\hat\om) , \theta_{\hat x}\rangle
= \langle\hat\om , \hat x\rangle$$
for all $\hat\om\in L_*^1(\hat\G)$, which yields,
by density of $L_*^1(\hat\G)$ in $L^1(\hat\G)$, that $\lam(\theta_{\hat x}) = \hat x$.
Conversely, assume that $\phi\in sp(\CU)$, then we have
\[
  \langle {\hat\lam}_u(\hat\om) , \theta_{\lam(\phi)}\rangle = 
\langle \hat\om , \lam(\phi)\rangle
= \langle{\hat\lam}_u(\hat\om) , \phi\rangle
\]
for all $\hat\om\in L_*^1(\hat\G)$. 
The density of ${\hat\lam}_u(L_*^1(\hat\G))$ in $\CU$
implies that $\theta_{\lam_u(\phi)} = \phi$.

Since $\lam : \CU^*\ra \LLL$ is an algebra homomorphism, the map 
$$sp(\CU)\ni\phi\mapsto\lam(\phi)\in sp(L^1(\hat\G))$$
defines a bijective group homomorphism.
Moreover, for a net $(\phi_\alpha)$ in $sp(\CU)$ and $\phi\in sp(\CU)$,
 the density of $L_*^1(\hat\G)$ and ${\hat\lam}_u(L_*^1(\hat\G))$ in
$L^1(\hat\G)$ and $\CU$, respectively, yield
\begin{eqnarray*}
\phi_\alpha\xrightarrow{w^*}\phi &\Longleftrightarrow&
\langle {\hat\lam}_u(\hat\om) , \phi_\alpha\rangle \longrightarrow \langle
 {\hat\lam}_u(\hat\om) , \phi\rangle
\ \forall \hat\om\in L_*^1(\hat\G) \\ &\Longleftrightarrow&
\langle \hat\om , \lam(\phi_\alpha)\rangle \longrightarrow \langle \hat\om , \lam(\phi)\rangle
\ \forall \hat\om\in L_*^1(\hat\G) \Longleftrightarrow
\lam(\phi_\alpha)\xrightarrow{w^*}\lam(\phi). 
\end{eqnarray*}
Hence, the map $sp(\CU)\ni\phi\mapsto\lam(\phi)\in sp(L^1(\hat\G))$ is a group homeomorphism.
\end{proof}

Next theorem combines all the above identifications.
\begin{theorem}\label{322} The following can be identified as locally compact groups:
\begin{enumerate}
 \item \ $\tilde\G$ with the strong operator topology;
 \item \ $Gr (\hat\G)$ with the weak$^*$ topology;
 \item \ $sp (L^1(\hat\G))$ with the weak$^*$ topology;
\item \ $Gr (\B(\LT),\hat\Gamma)$ with the weak$^*$ topology; 
\item \ $sp (\T_\star(\hat\G))$ with the weak$^*$ topology;
  \item \ $sp (\CU)$ with the weak$^*$ topology.
\end{enumerate}
\end{theorem}
\begin{proof} Since the spectrum of a Banach algebra is locally compact
 with weak$^*$ topology, all the above groups are locally compact groups.
\end{proof}
\begin{remark}
Applying Theorem \ref{322} to the case where $\G = VN(G)$ for a locally compact group
$G$, we obtain a generalization of a Renault's result (cf. \cite[Theorem 2]{Ren}),
in which $G$ is assumed amenable.
\end{remark}
\begin{theorem} \label{323} The assignment $\G\ra\tilde\G$ preserves 
compactness, discreteness, and hence finiteness.
\end{theorem}
\begin{proof} Let $\G$ be compact. Then $\hat\G$ is discrete, and
in view of Theorem \ref{322}, we may equivalently show that $Gr(\hat\G)$ is compact.
Let $\hat e\in L^1(\hat\G)$ be the unit.
Then, for any $\hat x \in Gr(\hat\G)$, we have
\begin{eqnarray*}
 \langle\hat f,\hat x\rangle  =  \langle\hat f\star\hat e,\hat x\rangle 
= \langle\hat f\ot\hat e,\hat\Gamma(\hat x)\rangle 
 =  \langle\hat f\otimes\hat e,\hat x\otimes\hat x\rangle = 
\langle\hat f,\hat x\rangle\langle\hat e,\hat x\rangle 
\end{eqnarray*}
for all $\hat f\in L^1(\hat\G)$. So $\langle\hat e,\hat x\rangle = 1$ for all $\hat x\in\hat\G$
and since $Gr(\hat\G) = sp(L^1(\hat\G))$ by Theorem \ref{318}, this implies that
the constant function $\hat e|_{Gr(\hat\G)} \equiv 1$ lies in $C_0(Gr(\hat\G))$.
Therefore $Gr(\hat\G)$ is compact.

Now let $\G$ be discrete. Then $\hat\G$ is compact, and again by Theorem \ref{322}
we need to show that $Gr(\hat\G)$ is discrete. Let $\hat x\in Gr(\hat\G)$ and
$\hat\fee\in L^1(\hat\G)$ be the Haar state.
Then we have
$$\langle\hat f,1\rangle\langle\hat\fee,\hat x\rangle =
 \langle\hat f \star\hat \fee,\hat x\rangle 
= \langle\hat f\otimes\hat\fee,\hat x\otimes\hat x\rangle = 
\langle\hat f,\hat x\rangle\langle\hat \fee,\hat x\rangle$$
for all $\hat f\in L^1(\hat\G)$. So if $\hat x\neq 1$, 
we must have $\langle\hat\fee,\hat x \rangle = 0$,
and since $\langle\hat\fee,1\rangle = 1$,
we see that $\hat\fee$, as a function on $Gr(\hat\G)$ is the characteristic function
of $\{1\}$. But since $\hat\fee$ is continuous on $Gr(\hat\G)$, the latter must be discrete.
\end{proof}

In the following (Theorem \ref{325}), we shall investigate the relation between the operations
$\G\ra\tilde\G$ and $\G\ra\hat\G$.
\begin{lemma}\label{324}  Let $G$ and $H$ be two locally compact groups in duality,
 i.e., there exists a continuous bi-homomorphism
$\langle \cdot , \cdot \rangle : G\times H\ra\field T$, where $\field T$
denotes the unit circle.
Define the sets
\[\begin{array}{c}
G_{1}:=\{g\in G: \langle g,H\rangle = 1\},\\
H_{1}:=\{h\in H: \langle G,h\rangle = 1\}.
\end{array}
\]
Then $G_{1}$ and $H_{1}$ are closed normal subgroups of $G$ and $H$, containing the
commutator subgroups, and we have
$$\frac{G}{G_{1}}\,\cong\,\widehat{\left(\frac{H}{H_{1}}\right)}.$$
\end{lemma}
\begin{proof} For all $g\in G$, $g_{1}\in G_{1}$ and $h\in H$ we have
$$\langle g^{-1}g_{1}g,h\rangle = \langle g^{-1},h\rangle
\langle g_{1},h\rangle\langle g,h\rangle = 
\langle g^{-1},h\rangle\langle g,h\rangle = \langle e,h\rangle = 1.$$
Therefore, $G_{1}$ is normal in $G$.
For all $g_{1},g_{2}\in G$ and $h\in H$ we have
$$\langle g_{1}g_{2}g_{1}^{-1}g_{2}^{-1},h\rangle =
\langle g_{1},h\rangle\langle g_{2},h\rangle
\langle g_{1}^{-1},h\rangle\langle g_{2}^{-1},h\rangle = 1.$$
Thus, $[G,G]\subseteq G_{1}$; 
similarly, we see that $H_1$ is normal in $H$, and $[H,H]\subseteq H_{1}$.
Now it just remains to show the last assertion. Define
$$\phi:G\rightarrow\widehat{\left(\frac{H}{H_{1}}\right)}, 
 \phi(g)(\overline{h}) = \langle g,h\rangle.$$
The definition of $H_1$ implies that $\phi(g)$ is
well-defined for each $g\in G$. 
Obviously, $\phi$ is a group homomorphism, and we have
$Ker (\phi)=G_{1}$. Hence we have an injective group homomorphism
$$\overline{\phi}:\frac{G}{G_{1}}\hookrightarrow\widehat{\left(\frac{H}{H_{1}}\right)}.$$
Similarly, by exchanging the roles of $G$ and $H$ we obtain
$$\frac{H}{H_{1}} \hookrightarrow \widehat{\left(\frac{G}{G_{1}}\right)}$$
whence
$$\widehat{\widehat{\left(\frac{{G}}{G_{1}}\right)}}
 \twoheadrightarrow\widehat{\left(\frac{H}{H_{1}}\right)}.$$
If we compose the last surjection with the identification of
$\frac{G}{G_{1}}$ with its second dual, we get
$\overline{\phi}$. Hence $\overline{\phi}$ is onto, and
$$\frac{G}{G_{1}} \,\cong\, \widehat{\left(\frac{H}{H_{1}}\right)}.$$
\end{proof}
It follows from Theorem \ref{317.5} and Lemma \ref{324} that
$Gr(\G)\cap Gr(\hat \G)'$ is a normal subgroup of $Gr(\G)$, and
\[
\frac{Gr(\G)}{Gr(\G)\cap Gr(\hat \G)'}\
\]
is an abelian group.  In the following we denote this group by $\tilde\G_1$.
\begin{theorem} \label{325}
Let $\G$ be a locally compact quantum group. Then we have a group homeomorphism
$${\hat{\tilde\G}}_1 \cong {\tilde{\hat\G}}_1 \ .$$
\end{theorem}
\begin{proof}
By Theorem \ref{317.5} we have a duality between $\tilde\G$ and ${\tilde {\hat{\G}}}$.
 Hence, theorem follows from Lemma \ref{324}.
\end{proof}

\par
In the above, to a locally compact quantum group $\G$,
we have assigned the locally compact group $\tilde\G$,
 which is easily seen to be an
invariant for $\G$. We shall now assign another invariant to $\G$.

Let $\G$ be a locally compact quantum group,
$v \in G r (\G)$ and $\hat{v} \in G r (\hat{\G})$. By Theorem \ref{317.5},
there exists $\lambda_{v,\hat{v}} \in \field T$ such that
$v \hat{v} = \lambda_{v,\hat{v}} \hat{v} v$.
It is easy to see that $(v,\hat v)\mapsto\lambda_{v,\hat{v}}$
defines a bi-homomorphism $\gamma: G r (\G) \times G r (\hat{\G}) \rightarrow \field T$.
\begin{proposition}
Let $\G$ be a locally compact quantum group. 
Then $Im(\gamma)$, the image
of $\gamma$, is a subgroup of $\field T$, and an invariant for $\G$.
\end{proposition}
\begin{proof} Using the above notation, we have a 
bi-character $\gamma_1 \,:\, {\tilde\G}_1 \times {\tilde{\hat\G}}_1
\rightarrow \field T$, induced by $\gamma$, with $I m (\gamma_1) = I m (\gamma)$ 
(see the proof of Lemma \ref{324}). 
Since ${\tilde\G}_1$ and ${\tilde{\hat\G}}_1$ are abelian, 
by the universal property of the
tensor product, there exists a homomorphism $\gamma_2 : {\tilde\G}_1
\otimes_{\field{Z}} {\tilde{\hat\G}}_1 \rightarrow \field T$, with $I m
(\gamma_2) = I m (\gamma_{1})$.
But $I m (\gamma_{2})$ is a
subgroup of $\field T$ since $\gamma_2$ is a group homomorphism.

\end{proof}

Since there is a good classification of subgroups of $\field T$ 
(cf. \cite[Theorem 25.13]{H-R}),
this invariant may be helpful towards some sort of classification of 
locally compact quantum groups.

\section{Examples}
In this section we calculate the locally compact group
$\tilde\G$ for some of the most interesting and well-known examples of non-classical,
non-Kac, locally compact quantum groups.

\subsection{Woronowicz's Compact Matrix Pseudogroups}
Let $A$ be a $C^*$-algebra with unit, $U_{N}=[u_{ij}]$
an $N\times N$ ($N\in\field N$) matrix with entries
belonging to $A$, and $\A$ be the $^*$-subalgebra of $A$
generated by the entries of $U_{N}$. Then $\mathbf {G} = (A,U_{N})$
is called a {\it compact matrix pseudogroup} \cite[Definition 1.1]{Wor1}
if the following hold:
\begin{enumerate}\item
 $\A$ is dense in $A$;
\item  there exists a $C^*$-homomorphism
$\Gamma$ from $A$ to $A \mtp A$ such that
$$ \Gamma(u_{ij}) = \sum_{k=1}^N u_{ik}\ot u_{kj} \ \  \ \ (i,j = 1,2,...,N);$$
\item  there exists a linear anti-multiplicative map
$\kappa : \A \ra \A$ such that
$$\kappa(\kappa(a^*)^*) = a$$
for all $a\in A$, and
\[\begin{array}{lll}
\sum_k \kappa(u_{ik})u_{kj} & = & \delta_{ij}1,\\
\sum_k u_{ik}\kappa(u_{kj}) & = & \delta_{ij}1.
\end{array}
\]
for all $i,j = 1,2,...,N$.
\end{enumerate}
\begin{theorem}\label{351}
 Let $\mathbf {G} =(A,U_{N})$ be a compact matrix pseudogroup. 
Then $\tilde\G$ is homeomorphic 
 to a compact subgroup of $GL_N(\C)$, hence a compact Lie group.
\end{theorem}
\begin{proof}  Define the map
$$\Phi:sp(A)\rightarrow M_N(\C) \ \ , \ \ f\mapsto [f(u_{ij})]_{ij}.$$
Then $\Phi$ is injective since $\{u_{ij}\}$ generates $A$.
For all $f,g\in sp(A)$, we have
\begin{eqnarray*}
\Phi(f\star g) &=& [f\star g(u_{ij})]_{ij}
= [(f\otimes g)\Gamma(u_{ij})]_{ij}
= [(f\otimes g)\sum_{k=1}^N(u_{ik}\otimes u_{kj})]_{ij}\\
&=& [\sum_{k=1}^Nf(u_{ik})g(u_{kj})]_{ij}
= [f(u_{ij})]_{ij}[g(u_{ij})]_{ij}
= \Phi(f)\Phi(g).
\end{eqnarray*}
So, $\Phi$ is an injective group homomorphism. Obviously $Im(\Phi)\subseteq GL_N(\C)$.

Since each of the maps $f\mapsto f(u_{ij})$ is continuous, $\Phi$ is also continuous.
By Theorem \ref{323}, $\tilde\G$ is compact, 
and therefore $\Phi$ is a homeomorphism onto its image.
\end{proof}

\subsection{$\mathbf{SU_{\mu}(2)}$}\label{sumu}
In the following, we consider Woronowicz's   twisted $SU_q(2)$   quantum group  
for  $q\in (-1, 1)$ and $q\neq 0$ (cf. \cite{Wor5}).
The quantum group  $SU_q(2)$ is a co-amenable compact matrix pseudogroup 
with the quantum group $C^*$-algebra $C(SU_q(2))= C_{u}(SU_{q}(2))$ generated by 
two operators $u$ and $v$ such that the matrix  
 $$U  = \left [ \begin{array}{cc}
  u & -qv^{*}  \\
  v & u^{*} 
\end{array} \right] $$
  is a unitary matrix in $M_{2}(C(SU_q(2)))$.
  
\begin{theorem}
Let $\G = SU_{\mu}(2)$. Then $\tilde\G = \field T$.
\end{theorem}
\begin{proof} Let $\Phi: sp(C(SU_{\mu}(2))) \rightarrow GL_{2}(\C)$
be as in the proof of Theorem \ref{351}, and $f\in sp(C(SU_{\mu}(2)))$.
It is easy to verify that $\field T\subseteq Im(\Phi)$, 
under the identification of $\field T$ with the matrices of the form
\renewcommand{\baselinestretch}{1}\normalsize
\[
\Phi(f) =  \left( \begin{array}{cc}
  \lambda & 0  \\
  0 & \bar{\lambda} 
\end{array} \right) \]
where $\lambda\in\field T$. To see  that any element in $Im(\Phi)$ is of this form,
note that $\Phi(f)$ is a unitary matrix in $M_{2}(\C)$ for all $f\in sp(C(SU_{\mu}(2)))$.
But since
\renewcommand{\baselinestretch}{1}\normalsize
\[
\Phi(f) =  \left( \begin{array}{cc}
  f(u) & -q\overline{f(v)}  \\
  f(v) & \overline{f(u)} 
\end{array} \right)\, , \]
it follows that $f(v)=0$, and theorem follows.
\end{proof}
%\doublespacing
%\renewcommand{\baselinestretch}{1.5}\normalsize

\subsection{$\mathbf{E_{\mu}(2)}$ and its Dual}
Let $(e_{k,l})_{k,l\in \Z}$ be the canonical basis for $l^2(\Z \times \Z)$.
Define operators $v$ and $n$ on $l^2(\Z \times \Z)$ as follows:
\[
\left\{ \begin{array}{lll}
ve_{k,l} & = & e_{k-1,l} \\
ne_{k,l} & = & {\mu}^k e_{k,l+1}\, .
\end{array}
\right.\]
Then $v$ is a unitary and $n$ is a normal operator with
$sp(n)\subseteq \overline{C}^\mu$, where
$$\overline{\C}^\mu:=\{z\in\C : z=0 \ \text{or} \ |z|\in\mu^\Z\}.$$
The ($C^*$-algebraic) locally compact quantum group
$E_{\mu}(2)$ is defined (cf. \cite[Section 1]{Wor4}) to be the non-unital
$C^*$-algebra generated by the operators of the form
$\Sigma v^k f_{k} (n)$, where $k$ runs over a
finite set of integers, and $f_{k} \in
C_0(\overline{\C}^\mu)$.
The co-multiplication $\Gamma$ is defined on $E_\mu(2)$ in the following way:
\[\left\{\begin{array}{lll}
\Gamma(v) & = & v\otimes v \\
\Gamma(n) & = & v\otimes n + n\otimes v\, .
\end{array}\right.\]
Using \cite[Theorem 1.1]{Wor4}, we can calculate $\tilde\G$ for the
quantum group $E_\mu(2)$.
\begin{theorem}
Let $\G = E_\mu(2)$. Then $\tilde\G \cong \field T $ .
\end{theorem}
\begin{proof}
In view of \cite[Theorem 1.1]{Wor4}, for any $z\in \field T$,
we can define $f_z\in sp(C_0(E_\mu(2)))$ such
that $f_z(v)=z$ and $f_z(n)=0$.
Conversely, for each $f\in sp(C_0(E_\mu(2)))$,  by
 \cite[Theorem 1.1]{Wor4} again, $f(v)$ 
is a unitary and $f(n)$ is a normal operator on $\C$, 
and we have $f(n)=f(v)f(n)\overline{f(v)}=\mu f(n)$. 
Since $\mu\neq 1$, we must have $f(n)=0$. Put $z:=f(v)\in\field T$. So $f=f_z$.
Moreover, for $z,z'\in \field T$ we have
\begin{eqnarray*}
f_z\star f_{z'}(v) &=& (f_z\otimes f_{z'})\Gamma(v) = (f_z\otimes f_{z'})(v\otimes v)\\
&=& f_z(v)f_{z'}(v) = zz' = f_{zz'}(v), \\
f_z\star f_{z'}(n) &=& (f_z\otimes f_{z'})\Gamma(n) = 
(f_z\otimes f_{z'})(v\otimes n+n\otimes v)\\
&=& f_z(v)f_{z'}(n)+f_z(n)f_{z'}(v) = 0 = f_{zz'}(n).
\end{eqnarray*}
Hence $f_z\star f_{z'}=f_{zz'}$.
\end{proof}

In \cite{Wor4} Woronowicz has also described $\widehat{E_\mu(2)}$, 
the dual quantum group of $E_\mu(2)$.
Similarly to $E_\mu(2)$, this quantum group is determined by 
two operators $N$ and $b$,
with co-multiplication determined by $\hat\Gamma(N) = N\ot 1 + 1\ot N$,
$\hat\Gamma(b) = b\ot \mu^{\frac{N}{2}} + \mu^{\frac{-N}{2}}\ot b$. 
The dual quantum group $\widehat{E_\mu(2)}$ has a universal property as well,
which makes it easy to calculate our group.
\begin{theorem}\label{emuhat}
Let $\G = \widehat{E_\mu(2)}$. Then $\tilde\G\cong \field Z $ .
\end{theorem}
\begin{proof} For any $s\in \Z$ we define $\hat f_s\in sp(C_0(\widehat{E_\mu(2)}),\C)$
 by $\hat f_s(b)=0$, $\hat f_s(N)=s$.
Then it is clear from \cite[theorem 3.1]{Wor4} 
that the map $s\mapsto \hat f_s\in sp(C_0(\widehat{E_\mu(2)}))$
is a bi-continuous bijection.
For $s,s'\in \Z$ we have
\begin{eqnarray*}
\hat f_s\star \hat f_{s'}(b)&=&(\hat f_s\otimes \hat f_{s'})\hat\Gamma(b) = 
(\hat f_s\otimes \hat f_{s'})(b\otimes \mu^{\frac{N}{2}}+\mu^{-\frac{N}{2}}\otimes b) \\
&=&  \hat f_s(b)\hat f_{s'}(\mu^{\frac{N}{2}})+\hat f_{s}(\mu^{-\frac{N}{2}})\hat f_{s'}(b) = 0 
= \hat f_{s+s'}(b), \\
\hat f_s\star \hat f_{s'}(N) &=& (\hat f_s\otimes \hat f_{s'})\hat\Gamma(N) =
(\hat f_s\otimes \hat f_{s'})(N\otimes 1+1\otimes N) \\
&=& \hat f_s(N)\hat f_{s'}(1)+\hat f_s(1)\hat f_{s'}(N) = s+s' 
= \hat f_{s+s'}(N).
\end{eqnarray*}
So $\hat f_s\star\hat f_{s'}=\hat f_{s+s'}$.
\end{proof}

\section{Structural Properties of $\G$ encoded by $\tilde\G$}
In this section, we investigate the relation between the structure of
$\G$ and that of $\tilde\G$.

\subsection{Unimodularity of $\G$}
Let $\G$ be a locally compact quantum group.
The main goal of this section (Theorem \ref{336.5}) is to show that if
both $\tilde\G$ and $\tilde{\hat\G}$ are small, then $\G$ is of a
very specific type, namely a unimodular Kac algebra.

In the sequel, for a locally compact group $G$, we denote by $\mathbf Z(G)$
the center of the group $G$.
\begin{proposition}\label{332}
Let $\G$ be a locally compact quantum group.
If $\mathbf Z(Gr(\G))$ is discrete, then $\G$ is unimodular.
\end{proposition}
\begin{proof}
By \cite[Proposition 4.2.]{Baaj-Vaes} we have $\delta^{it} \in \mathbf Z(Gr(\G))$ 
for all $t \in\field{R}$, where $\delta$ is the modular element of $\G$,
and $\mathbf Z(Gr(\G))$ is the center of the intrinsic group $Gr(\G)$.
Since the map
\[
 \R\ni t\mapsto\delta^{it}\in \mathbf Z(Gr(\G))
\]
is continuous, its range must be connected.
But since $\mathbf Z(Gr(\G))$ is discrete, the range must be a single point.
Therefore, we obtain $\delta^{it} = 1$ for all $t \in \field{R}$, which implies $\delta = 1$.
\end{proof}

Combining Proposition \ref{332} with Theorem \ref{emuhat}, we obtain the following.
\begin{corollary}
The quantum group $E_\mu(2)$ is unimodular.
\end{corollary}

\begin{lemma}\label{333}
Let $\Phi$ and $\Psi$ be weak$^*$ continuous linear maps on
$\LL$. If we have 
$(\iota \otimes \Phi)\circ \Gamma = (\iota \otimes \Psi)\circ \Gamma$ or  
$(\Phi \otimes \iota)\circ \Gamma = (\Psi \otimes \iota)\circ \Gamma$,
then $\Phi = \Psi$.
\end{lemma}
\begin{proof} Assume that $(\iota \otimes \Phi)\circ \Gamma = 
(\iota \otimes \Psi)\circ \Gamma$.
Then, for all $x \in \LL$ and $\omega \in \LO$, we have
$$\Phi ((\omega \otimes \iota) \Gamma (x)) = \Psi ((\omega \otimes \iota) \Gamma(x)).$$
Since the set $\{\,(\omega \otimes \iota) \Gamma (x) \,:\, \omega \in L^1(\G), x \in \LL\,\}$
is weak$^*$ dense in $\LL$, the conclusion follows.
The argument assuming the second relation is analogous.
\end{proof}

\begin{lemma}\label{334}
If $\delta = 1$ and $\sigma^\fee_{t} = \tau_{t}$ for all $t \in \field{R}$,
then $\tau_{t} = \sigma^\fee_{t} = \iota$ for all $t\in\R$,
and $\G$ is a Kac algebra.
\end{lemma}
\begin{proof} Since $\delta = 1$, we have $\sigma^\psi_{t} = \sigma^\fee_{t}$ 
for all $t\in\R$. 
Moreover, since $\sigma^\fee_{t} = \tau_{t}$, by the 
equations (\ref{1.5}) and (\ref{1.6}), we have
\begin{eqnarray*}
(\sigma^\fee_{t} \otimes \sigma^\psi_{-t})\circ \Gamma = \Gamma \circ\tau_{t}
= (\tau_{t} \otimes \tau_{t})\circ \Gamma 
= (\sigma^\fee_{t} \otimes \sigma^\psi_{t})\circ \Gamma,
\end{eqnarray*}
which implies that $(\iota \otimes \sigma^\psi_{-t}) \circ\Gamma 
= (\iota \otimes \sigma^\psi_{t})\circ \Gamma$
for all $t\in\R$.
Now, Lemma \ref{333} yields $\sigma^\psi_{-t} = \sigma^\psi_{t}$,
i.e., $\sigma^\psi_{2t} = \iota$, for all $t \in \field{R}$.
Hence, $\tau_{t} = \sigma^\fee_{t} = \sigma^\psi_{t} = \iota$ for all $t \in \field{R}$,
and therefore $\G$ is a Kac algebra.
\end{proof}

\begin{proposition}\label{335}
If $\G$ and $\hat{\G}$ are both unimodular, then $\G$ is a
Kac algebra.
\end{proposition}
\begin{proof} Since $\hat{\delta} = 1$, 
equations (\ref{126}) imply that $P^{it} = \Delta_\fee^{it}$,
whence $\sigma^\fee_{t} = \tau_{t}$
for all $t \in \field{R}$. Since, in addition, $\delta = 1$,
Lemma \ref{334} yields the claim.
\end{proof}

In particular, combining Propositions \ref{332} and \ref{335}, we see that
for a unimodular locally compact quantum group $\G$,
the smallness of the group $\tilde{\G}$ forces the
quantum group $\G$ to be of Kac type.
\begin{corollary}\label{336}
Let $\G$ be a unimodular locally compact quantum group. If $\mathbf Z(\tilde{\G})$
is discrete, then $\G$ is a Kac algebra.
\end{corollary}

Since every compact quantum group is unimodular, we obtain the following
interesting result.
\begin{corollary}\label{336.01}
Let $\G$ be a compact quantum group. 
If $\mathbf Z(\tilde\G)$ is discrete, then $\G$ is a Kac algebra.
\end{corollary}

The class of non-Kac compact quantum groups is one of the most
important and well-known classes of non-classical quantum
groups. Some important examples of such objects are deformations
of compact Lie groups, such as Woronowicz's famous
$SU_{\mu}(2)$, which we have discussed in Section \ref{sumu}. 
This shows the significance of our Corollary \ref{336.01}:

there is some richness of classical information in these classes of quantum structures!
\begin{theorem}\label{336.5}
 Let $\G$ be a locally compact quantum group. 
 If $\mathbf Z(\tilde\G)$ and $\mathbf Z(\tilde{\hat\G})$ are both
discrete, then $\G$ is a unimodular Kac algebra.
\end{theorem}
\begin{proof} Since $\mathbf Z(\tilde{\G})$ is discrete, 
$\hat{\G}$ is unimodular by Proposition \ref{332},
and hence our assertion follows from Corollary \ref{336}, applied to $\hat\G$.
\end{proof}

\subsection{Traciality of the Haar Weights}
There is a strong connection between the group $\tilde{\G}$,
assigned to $\G$, and traciality of the Haar weights, especially
in the Kac algebra case. Since the scaling group $\tau_t$ in this case
is trivial, equation (\ref{1.5}) implies
\begin{equation}\label{e337}
\Gamma \sigma^\fee_{t} = (\iota \otimes \sigma^\fee_{t})\Gamma \ \ \forall t\in\R,
\end{equation}
and so $\sigma^\fee_{t} \in {\mathcal{C}}\B^\sigma_{cov} (\LL)$.
Thus, we obtain $\Delta_\fee^{it} \in
Gr(\hat{\G})$ for all $t\in\R$ (this was also proved in \cite[Corollary 2.4]{De-Cann1}).
Hence, we see that if $\G$ is a Kac algebra such that $\tilde\G$
 is trivial, then $\fee$ is tracial.
Similarly to Proposition \ref{332}, the fact that the map $t \mapsto \Delta_\fee^{it}$ is
continuous, allows us to further generalize this result.
\begin{proposition}\label{339}
Let $\G$ be a Kac algebra. If $\tilde{\G}$ is discrete, then
$\varphi$ is tracial.
\end{proposition}

We have seen in Corollary \ref{336} that for a
 unimodular locally compact quantum group $\G$, the
smallness of $\tilde{\G}$ forces the quantum group to be a Kac
algebra. The situation is similar for traciality.
\begin{proposition}\label{341}
Let $\G$ be a locally compact quantum group with tracial Haar weight. If
$Gr(\G)$ is discrete, then $\G$ is a Kac algebra.
\end{proposition}
\begin{proof} Since $Gr(\G)$ is discrete, we have $\delta = 1$. Also,
traciality of $\varphi$ implies that
$$\Gamma = \Gamma\sigma^\fee_t = (\tau_t \otimes \sigma^\fee_t)\Gamma
 = (\tau_t \otimes\iota)\Gamma,$$
for all $t\in\R$, which implies that $\tau = \iota$, by Lemma \ref{333}.
Hence, $\G$ is a Kac algebra.
\end{proof}
Moreover, by combining \cite[Proposition 6.1.2]{E-S} with Theorem \ref{336.5}, 
we obtain a stronger version of the latter.
\begin{theorem}
Let $\G$ be a locally compact quantum group. 
If $\mathbf Z(\tilde\G)$ and $\mathbf Z(\tilde{\hat\G})$ are both
discrete, then $\G$ is a unimodular Kac algebra with tracial Haar weight.
\end{theorem}

\subsection{Amenability}
In the last part of this section we investigate the question of whether
amenability passes from $\G$ to $\tilde\G$. 
In the following, for a locally compact quantum group $\G$,
we denote by $N_\G$ the sub von Neumann algebra of $\LL$ generated by 
the intrinsic group $Gr(\G)$.
%\end{definition}
\begin{proposition}\label{343}
Let $\G$ be a locally compact quantum group such that 
the restriction $\fee_{|N_\G}$ is
semi-finite. Then we have the identification
\[
  N_\G \cong VN(Gr(\G)).
\]
\end{proposition}
\begin{proof} It is obvious that the restriction of
$\Gamma$ to $N_\G$ defines a co-multiplication on $N_\G$,
and by our assumption, on $N_\G$, the restriction of 
$\varphi$ to $N_\G$ is an n.s.f. left
invariant weight on $N_\G$.
Also, since $R(v) = v^*$ for all $v \in Gr(\G)$,
we have $R(N_\G) \subseteq N_\G$, which implies that $\varphi \circ R$
defines a right Haar weight on $N_\G$. Therefore, $N_\G$ can be given
a locally compact quantum group structure. Obviously, it is co-commutative, and so
$N_\G \,\cong\, V N (G r (\G))$, by \cite[Theorem 2]{Tak}.
\end{proof}

\begin{corollary}\label{346}
If $\G$ is a compact quantum group, then
$N_\G \cong VN(Gr(\G))$.
\end{corollary}
\begin{theorem}\label{347}
Let $\G$ be a discrete quantum group. If $\G$ is amenable, then
so is the group $\tilde{\G}$.
\end{theorem}
\begin{proof}
If $F\in\LL^*$ is an invariant mean on $\G$, then one can see that 
the map $(F\ot\id)\circ\Gamma : \B(\LT)\ra\B(\LT)$
is a conditional expectation on $\LLL$. Also,
 since $\sigma^{\hat\fee}_t(\hat x)\in\C \hat x$ for all $t\in\R$ and $\hat x\in Gr(\hat\G)$,
we have $\sigma_t^{\hat\fee}(N_{\hat\G})\subseteq N_{\hat\G}$, 
and therefore, by \cite[Theorem 4.2]{tak2}, there
exists a conditional expectation $E:\LLL\ra N_{\hat\G}$. Hence, $E\circ (F\ot\id)\circ\Gamma$
is a conditional expectation from $\B(\LT)$ on $N_{\hat\G} \cong VN(\tilde\G)$, and so
$VN(\tilde\G)$ is injective. Since , by Theorem \ref{323}, 
$\tilde\G$ is discrete, it follows that $\tilde\G$ is amenable (cf. \cite{Connes}).
\end{proof}

Let $i : N_\G \hookrightarrow L^\infty (\G)$ be the canonical
injection. Obviously, $i$ is weak$^*$ continuous, so we have the
pre-adjoint map $i_{*} : L^1 (\G) \twoheadrightarrow (N_\G)_{*}$.
Then direct calculation implies the following.
\begin{lemma}\label{344}
The map $i_{*} : L^1 (\G) \twoheadrightarrow (N_\G)_{*}$ is a 
completely bounded algebra homomorphism.
\end{lemma}

If a locally compact quantum group $\G$ satisfies the condition of Proposition \ref{343},
then $V N (G r (\G))\cong G r (\G)'' \cap L^\infty (\G)$, and
by the above, we have a surjective continuous algebra homomorphism $i_{*} : L^1 (\G)
\rightarrow A(Gr(\G))$. Therefore, in this case, many of the
algebraic properties of $L^1 (\G)$ will be satisfied by the
Fourier algebra of the intrinsic group as well.

There are many different equivalent characterizations of amenability
for a locally compact group. The question of whether the quantum counterpart
of these conditions are equivalent as well, remains unsolved in many important
instances.

In the following, we present a few of those equivalent characterizations in the
group case which can be found for instance in \cite{Volker}.
\begin{theorem}\label{017}
For a locally compact group $G$, the following are equivalent:
\begin{enumerate}
 \item $G$ is amenable;
 \item $L^1(G)$ is an amenable Banach algebra;
 \item $A(G)$ has a bounded approximate identity (BAI);
 \item $A(G)$ is operator amenable.
\end{enumerate}
\end{theorem}
The equivalence $(1)\Leftrightarrow(2)$ is due to Johnson (1972),
$(1)\Leftrightarrow(3)$ is Leptin's theorem (1968),
and $(1)\Leftrightarrow(4)$ is due to Ruan (1995).
\begin{proposition}\label{345}
Let $\G$ be a discrete quantum group. If $L^1(\hat\G)$ has a BAI, then $\tilde\G$ is amenable.
\end{proposition}
\begin{proof} Since $\hat\G$ is compact, by Corollary \ref{346}, we have
$N_\G \cong VN(Gr(\G))$.
If $(\hat\omega_{\alpha})$ is a BAI for $L^1 (\hat\G)$, then as easily seen
$(i_{*} (\hat\omega_{\alpha}))$ is a BAI for $A(Gr(\G))$,
whence $Gr(\G)$ is amenable, by Theorem \ref{017}.
\end{proof}
\begin{proposition}\label{019283}
Let $\G$ be a discrete quantum group. If $\LO$ is operator amenable,
then $\tilde\G$ is amenable.
\end{proposition}
\begin{proof} Since $\G$ is discrete, we have, by Lemma \ref{344},
a completely bounded surjective algebra homomorphism from $L^1(\hat\G)$ onto
$A(\tilde\G)$. Since $L^1(\hat\G)$ is operator amenable, then so is $A(\tilde\G)$.
Hence, by Theorem \ref{017}, $\tilde\G$ is amenable.
\end{proof}
\begin{proposition}
Let $\G$ be a discrete quantum group. If $L^1(\hat\G)$ is an amenable Banach algebra,
then $\tilde\G$ is almost abelian.
\end{proposition}
\begin{proof} The argument is analogous to the one given in Proposition \ref{019283},
but instead of Theorem \ref{017}, we use \cite[Theorem 2.3]{For-Run},
stating that $A(G)$ is amenable if and only if $G$ is almost abelian.
\end{proof}


\begin{thebibliography}{}



\bibitem {BS} Baaj, S., Skandalis, G.: 
Unitaires multiplicatifs et
dualit\'{e} pour les produits crois\'{e}s de
$C^*$-alg\`ebres. Ann. Sci. \'Ecole Norm. Sup.
\textbf{26}, 425--488 (1993)



\bibitem {Baaj-Vaes} Baaj, S., Vaes, S.:
 Double crossed products of locally compact quantum groups. J. Inst. Math. Jussieu
\textbf{4},  no. 1, 135--173 (2005)






\bibitem {B-T} Bedos, E., Tuset, L.:
 Amenability and co-amenability for locally compact quantum groups.  Internat. J. Math.
\textbf{14}, 865--884 (2003)





\bibitem {Connes} Connes, A.:
 Classification of injective factors.  Ann. of Math. (2) 
 \textbf{104}, no. 1, 73--115 (1976)







\bibitem {De-Cann1} de Canni\`{e}re, J.:
 On the intrinsic group of a Kac algebra. Proc. London Math. Soc. (3)
\textbf{40}, no. 1, 1--20 (1980)








\bibitem {Dec-Van} De Commer, K.,  Van Daele, A.:
 Algebraic quantum groups imbedded in $C^*$-algebraic quantum groups. Rocky Mountain J. Math.
\textbf{40}, no. 4, 1149--1182 (2010)





\bibitem {E-R}  Effros, E. G.,  Ruan, Z.-J.:
 Operator Spaces. London Mathematical Society Monographs, New Series, 23,
The Clarendon Press, Oxford University Press, New York (2000)



\bibitem {E-S} Enock, M.,  Schwartz, J. M.:
 Kac Algebras and Duality of Locally Compact Groups. Springer-Verlag, Berlin (1992)









\bibitem {For-Run} Forrest, B., Runde, V.:
 Amenability and weak amenability of the Fourier algebra. 
Math. Z. \textbf{250}, no. 4, 731--744 (2005)




%\bibitem {Haag-op.val1} U. Haagerup,
%{\it Operator-valued weights in von Neumann algebras. I,} J. Funct. Anal.
%\textbf{32} (1979), no. 2, 175--206.

%\bibitem {Haag-op.val2} U. Haagerup,
%{\it Operator-valued weights in von Neumann algebras. II,} J. Funct. Anal.
%\textbf{33} (1979), no. 3, 339--361.


\bibitem {Haag-stan} Haagerup, U.:
 The standard form of von Neumann algebras. Math. Scand.
\textbf{37}, no. 2, 271--283 (1975)



\bibitem {H-R}  Hewitt, E.,  Ross, K. A.:
 Abstract Harmonic Analysis. Vol. I. Structure of Topological Groups, Integration Theory, 
Group Representations. Second edition,
 Grundlehren der Mathematischen Wissenschaften
  [Fundamental Principles of Mathematical Sciences], 
115, Springer-Verlag, Berlin-New York (1979)


\bibitem {HNR} Hu, Z.,  Neufang, M.,  Ruan, Z.-J.:
 Completely bounded multipliers over locally compact quantum groups.
 Proc. London Math. Soc. (to appear)



\bibitem {J-N-R}  Junge, M., Neufang, M.,  Ruan, Z.-J.:
 A representation theorem for locally compact quantum groups. Internat. J. Math.
\textbf{20}, no. 3, 377--400 (2009)

\bibitem {thesis}  Kalantar, M.:
Towards Harmonic analysis on Locally compact quantum groups:
From Groups to Quantum Groups and Back.  Ph.D. thesis,
Carleton University, Ottawa (2010)

\bibitem {kus} Kustermans, J.:
 Locally compact quantum groups in the universal setting. Internat. J. Math.
\textbf{12},  no. 3, 289--338  (2001)


\bibitem {K-V} Kustermans, J., Vaes, S.:
 Locally compact quantum groups. Ann. Sci. Ecole Norm. Sup.
\textbf{33},  837--934 (2000)



\bibitem {K-V-2} Kustermans, J., Vaes, S.:
 Locally compact quantum groups in the von Neumann algebraic setting. Math. Scand.  
\textbf{92}, no. 1, 68--92 (2003)






\bibitem {Muller} M\"{u}ller, A.:
 Classifying spaces for quantum principal bundles. Comm. Math. Phys.
\textbf{149},  no. 3, 495--512 (1992)



%\bibitem {M-ths} M. Neufang,
%{\it Abstrakte harmonische Analyse und Modulhomomorphismen \"{u}ber von
%Neumann-Algebren,} Dissertation zur Erlangung des Grades des Doktors der Naturwissenschaften,
%Saarbr\"{u}cken, 2000.
















\bibitem {Ren} Renaud, P. F.:
 Centralizers of Fourier algebra of an amenable group. Proc. Amer. Math. Soc.
\textbf{32}, 539--542  (1972)


\bibitem {Rua}  Ruan, Z.-J.:
 The operator amenability of $A(G)$.  Amer. J. Math.
\textbf{117}, no. 6, 1449--1474  (1995)







\bibitem {Volker} Runde, V.:
 Lectures on Amenability. Lecture Notes in Mathematics 1774,
 Springer-Verlag, Berlin (2002)




\bibitem {Tak} Takesaki, M.:
 A characterization of group algebras as a converse of 
Tannaka--Stinespring--Tatsuuma duality theorem. Amer. J. Math.
\textbf{91},  529--564 (1969)

\bibitem {tak2} Takesaki, M.:
 Theory of Operator Algebras. II. Encyclopaedia of Mathematical Sciences 125. 
Operator Algebras and Non-commutative Geometry, 6. Springer-Verlag, Berlin (2003)



%\bibitem {Tomat} R. Tomatsu,
%{\it Amenable discrete quantum groups,} J. Math. Soc. Japan
%\textbf{58} (2006),  no. 4, 949--964.






\bibitem {Vaes1} Vaes, S.:
 Locally Compact Quantum Groups. Ph.D. thesis,
K.U. Leuven (2000)



\bibitem {Wen} Wendel, J. G.: 
 Left centralizers and isomorphisms of group algebras. Pacific J. Math.,
 251--261 (1952)

\bibitem {Wor1} Woronowicz, S. L.:
 Compact matrix pseudogroups.  Comm. Math. Phys.
\textbf{111}, no. 4, 613--665 (1987)


\bibitem {Wor2} Woronowicz, S. L.: 
 Compact quantum groups. Sym\'{e}tries quantiques
(Les Houches, 1995), North-Holland, Amsterdam, 845--884  (1998)





\bibitem {Wor4} Woronowicz, S. L.: 
 Quantum $E(2)$ group and its Pontryagin dual.  Lett. Math. Phys.
\textbf{23},  no. 4, 251--263 (1991)



\bibitem {Wor5} Woronowicz, S. L.:
 Tannaka-Krein duality for compact matrix pseudogroups. 
Twisted $SU(N)$ groups.  Invent. Math.
\textbf{93}, no. 1, 35--76 (1988)





\end{thebibliography}
\end{document}